\documentclass{amsart}
\usepackage{amsmath, amssymb}

\input xy
\xyoption{all}

\def\cb{\operatorname{cb}}
\def\cbp{\operatorname{CBP}}
\def\acl{\operatorname{acl}}
\def\tp{\operatorname{tp}}
\def\stp{\operatorname{stp}}
\def\dcl{\operatorname{dcl}}
\def\ccm{\operatorname{CCM}}
\def\dcf{\operatorname{DCF}_0}
\def\PP{\mathbb{P}}
\def\CC{\mathbb{C}}
\def\C{\mathcal{C}}
\def\Int{\operatorname{Int}}
\def\loc{\operatorname{loc}}

\def\aut{\operatorname{Aut}}

\def\Ind#1#2{#1\setbox0=\hbox{$#1x$}\kern\wd0\hbox to 0pt{\hss$#1\mid$\hss}
\lower.9\ht0\hbox to 0pt{\hss$#1\smile$\hss}\kern\wd0}
\def\ind{\mathop{\mathpalette\Ind{}}}
\def\Notind#1#2{#1\setbox0=\hbox{$#1x$}\kern\wd0\hbox to 0pt{\mathchardef
\nn=12854\hss$#1\nn$\kern1.4\wd0\hss}\hbox to
0pt{\hss$#1\mid$\hss}\lower.9\ht0 \hbox to
0pt{\hss$#1\smile$\hss}\kern\wd0}
\def\nind{\mathop{\mathpalette\Notind{}}}

\newtheorem{theorem}{Theorem}[section]
\newtheorem{proposition}[theorem]{Proposition}
\newtheorem{lemma}[theorem]{Lemma}
\newtheorem{corollary}[theorem]{Corollary}
\newtheorem{claim}[theorem]{Claim}
\newtheorem{example}[theorem]{Example}

\theoremstyle{definition}
\newtheorem{definition}[theorem]{Definition}
\newtheorem{fact}[theorem]{Fact}

\theoremstyle{remark}
\newtheorem{remark}[theorem]{Remark}

\begin{document}

\title[Fibrations and algebraic reductions]{Some model theory of\\ fibrations and algebraic reductions}
\author{Rahim Moosa and Anand Pillay}
\date{December 17th, 2013}
\thanks{The first author was supported by an NSERC Discovery Grant.
The second author was supported by EPSRC grant EP/I002294/1.
Both authors would like to thank the Max Planck Institute in Bonn where some of this work was carried out.}

\thispagestyle{empty}

\bigskip
\bigskip

\begin{abstract}
Let $p= \tp(a/A)$ be a stationary type in an arbitrary  finite rank stable theory, and $\PP$ an $A$-invariant family of partial types.
The following property is introduced and characterised: whenever $c$ is definable over $(A,a)$ and $a$ is not algebraic over $(A,c)$ then $\tp(c/A)$ is almost internal to $\PP$.
The characterisation involves  among other things an apparently new notion of ``descent" for stationary types. 
Motivation comes partly from results in  Section~2 of [Campana, Oguiso, and Peternell.  Non-algebraic hyperk\"ahler manifolds. {\em Journal of Differential Geometry}, 85(3):397--424, 2010]
where structural properties of generalised hyperk\"ahler manifolds are given. 
The model-theoretic results obtained here are applied back to the complex analytic setting to prove that the algebraic reduction of a nonalgebraic (generalised) hyperk\"ahler manifold does not descend.
The results are also applied to the theory of differentially closed fields, where examples coming from differential algebraic groups are given.

\end{abstract}
\maketitle
\tableofcontents
\section{Introduction}

\noindent
A well-known result in stability theory is that any stationary type $p$ of finite $U$-rank  has an ``analysis" via semiminimal types.
More precisely if $p= \tp(a/A)$ then there are $a_{0},...,a_{n}$ such that $\dcl(A,a)=\dcl(A,a_{n})$, $a_{i}\in\dcl(A,a_{i+1})$, and all the types $\tp(a_0/A)$ and $\tp(a_{i+1}/A,a_{i})$ for $i<n$ are stationary and semiminimal.
Here the {\em seminimality} of a type $r\in S(B)$ means that for some stationary $U$-rank $1$ type $q$ nonorthogonal to $B$, $r$ is almost  internal to the family of $B$-conjugates of $q$.
This applies in particular to the many-sorted  theory $\ccm$ of compact complex manifolds, in a natural language (see~\cite{sat}).
If $p = \tp(a)$ is the generic type of an irreducible  compact complex manifold $X$ then  $a_{0},..,a_{n}$ will be  generic types of irreducible compact complex manifolds $X_{0},..,X_{n}$ and we have dominant meromorphic maps $X\to X_{n}\to X_{n-1} \to\cdots \to X_{0}$ such that $X_0$ has generic type semiminimal, each $X_{i+1}\to X_{i}$ for $i<n$ has general fibre with generic type semiminimal,  and $X\to X_{n}$ is a bimeromorphism.
In fact this is precisely the method used in bimeromorphic geometry to understand or describe a  (possibly nonalgebraic) K\"ahler manifold $X$.
The reason we restrict to K\"ahler here is that due to essential saturation (see~\cite{sat}), we may assume that $a$ and the $a_{i}$ are in the standard model so the fibres of the maps are also compact complex manifolds.
The usual procedure is to begin the analysis with the so-called {\em  algebraic reduction} of $X$, namely taking $X_{0}$ to be an algebraic (i.e. Moishezon or even projective) compact complex manifold of maximal dimension which is an image of $X$ under a dominant meromorphic map.
Then $X_{0}$  (and $\tp(a_{0})$)  will be internal to the strongly minimal sort of the projective line $\PP(\CC)$.
The dimension of $X_{0}$ is called the algebraic dimension $a(X)$ of $X$. 

A compact complex manifold $X$ is said to be hyperk\"ahler if it has dimension $2n$ and the space of homolomorphic $2$-forms is spanned by a $2$-form $\sigma$  such that $\sigma^{n}$ is everwhere nonzero.
If we relax the condition to ask only that $\sigma^n\neq 0$, then we get the class of generalised hyperk\"ahler manifolds.
In the paper ~\cite{cop} Campana, Oguiso and  Peternell raise a conjecture concerning hyperk\"ahler manifolds $X$:
{\em If $X$ is not algebraic then $a(X)$ is either $0$ or $n$, and in the latter case there is a holomorphic model $f:X\to B$ of the algebraic reduction such that $f$ is Lagrangian, namely the restriction of $\sigma$ to a general fibre of $f$ is zero.} 
The main part of~\cite{cop} is devoted to providing evidence for this conjecture, in particular proving it in dimension $4$.
These results make heavy use of certain structural results for generalised hyperk\"ahler manifolds in section 2 of that  paper, and it is this section 2 which throws up the notions that we study here in a general model-theoretic framework.

Theorem 2.3(1) of ~\cite{cop} says that a generalised hyperk\"ahler manifold $X$ has the following property:
\begin{itemize}
\item[(1)]
If $h:X\to Y$ is any fibration (namely dominant meromorphic map with general fibre connected), then $Y$ is algebraic, namely $h$ factors through the algebraic reduction $f:X\to B$. 
\end{itemize}
This can be rewritten model-theoretically as: suppose $p$ is the generic type of $X$, $a$ realizes $p$, and $c\in\dcl(a)$.
Then either $a\in\acl(c)$ or $tp(c)$ is internal to $\PP(\CC)$.
This is precisely the property that we will study in an arbitrary finite rank context, replacing ``internal to $\PP(\CC)$" by ``almost internal to $\PP$"  for some given invariant family $\PP$ of partial types.
We express the property as: {\em all proper fibrations of $p$ are over $\PP$}.
Let us remark here that in $\ccm$, internality to $\PP(\CC)$ is equivalent to almost internality to $\PP(\CC)$. 

It follows immediately from~(1) that:
\begin{itemize}
\item[(2)]
The algebraic reduction $f:X\to B$ is a ``minimal fibration"  in the sense that one cannot write $f$ as a composition of proper fibrations $g:X\to S$ and $h:S \to B$.
\end{itemize}
Suppose that $a$ realizes the generic type of $X$ and $f(a) = b$.
The content of~(2) is that if $c\in\dcl(a,b)\setminus\acl(b)$ then $a\in\acl(b,c)$.
In the general context of stable theories of finite rank we will say that a complete stationary type $p = \tp(a)$ {\em has no proper fibrations} if the above condition holds.

So, in the general finite rank context, if all proper fibrations of a stationary type $p=\tp(a)$ are over a given invariant set of partial types $\PP$, and $b$ is the maximal tuple in $\dcl(a)$ such that $\tp(b)$ is almost internal to $\PP$ (the analogue of the algebraic reduction), then $\tp(a/b)$ has no proper fibrations.
The question is: what property do we have to add to obtain the converse?
We find an answer in Theorem~\ref{characterisepfibre}:  $\tp(a/b)$ must not ``descend" to any proper relatively algebraically closed subset of $\dcl(b)$.
The notion of descent here is natural and has clear geometric content when interpreted in $\ccm$.
In Section~\ref{ccm-section} we perform this specialisation back to the complex analytic setting, recovering and improving the results on generalised hyper\"ahler manifolds that motivated this work (see Corollary~\ref{hyperkaehler}).

The following condition is also proved of generalised hyperk\"ahler $X$ as part of Theorem~2.1 in~\cite{cop}, and is deduced there from~(2) by complex analytic arguments: 
\begin{itemize}
\item[(3)]
The general fibre of the algebraic reduction  $f:X\to B$ is either algebraic, in which case it is an abelian variety, or of algebraic dimension zero in which case it is ``isotypically semi-simple".
\end{itemize}
Isotypically semisimple means in generic finite-to-finite correspondence with some $S^{k} = S\times ... \times S$ where $S$ is a simple compact complex manifold; ``simple" here means precisely that the generic type of $S$ has $U$-rank $1$. 
For this too we give a model-theoretic account, in Section~\ref{nofibre-section}.
Firstly, we show that a stationary type $p$ in an arbitrary finite rank stable theory has no proper fibrations if and only if it is either almost internal to a non locally modular minimal type (the analogue of an algebraic manifold) or it is interalgebraic over the base parameters with a power of a locally modular minimal type (the analogue of isotypically semi-simple).
This is Proposition~\ref{nofibrations}, a result probably known implicitly to experts in the field.
However, the part of~(3) saying that when the general fibre of $f$ is algebraic it is in fact an abelian variety, is not covered by~\ref{nofibrations}.
Indeed, our account of this refinement, which appears as Proposition~\ref{moish-abelian} below, while essentially model-theoretic (using definable automorphism groups) does make use of a complex analytic argument in one place.

Our results can be specialised to theories other than $\ccm$, and this is of course a large part of the point.
In Section~\ref{dcf-section} we consider differentially closed fields of characteristic $0$.
Here the role of $\PP$ is played by the field of constants $\C$.
So a finite dimensional differential-algebraic variety is considered ``algebraic" if it is almost internal to the constants, we say {\em $\C$-algebraic}.
Proposition~\ref{minimaldcf} and Theorem~\ref{characterisepfibredcf} can then be viewed as differential-algebraic geometric analogues of results from complex-analytic geometry.
In Example~\ref{example} we illustrate a class of non $\C$-algebraic finite dimensional differential-algebraic varieties (indeed differential-algebraic groups) to which these theorems apply.

Other contexts to which these results would apply, but that we do not pursue here, are finite dimensional difference varieties as well as Hasse-Schmidt varieties in positive characteristic.

We should mention that the distinction between definable and algebraic closure is crucial in this paper. 

\bigskip
\section{No Fibrations}
\label{nofibre-section}

\noindent
In the next two sections we work in the general setting of a sufficiently saturated model $\overline M^{\operatorname{eq}}$ of a complete theory $T$ of finite $U$-rank.
Our stability-theoretic notation is standard, see ~\cite{pillay96}.
The word ``minimal" is used in this paper in various different ways, but by a minimal type we mean a stationary type of $U$-rank $1$. 

\begin{definition}
\label{nofibrationsdef}
A stationary type $\tp(a/A)$ {\em admits no proper fibrations} if whenever $c\in \dcl(Aa)\setminus\acl(A)$ then $a\in\acl(Ac)$.
\end{definition}

Clearly any minimal type admits no proper fibrations.
But there exist others.

\begin{example}
\label{examplenofibre}
Suppose $p\in S(A)$ is a trivial minimal stationary type.
If $a_1,\dots,a_r$ are independent realisations of $p$ and $a$ is a code for $\{a_1,\dots,a_r\}$, then $\tp(a/A)$ admits no proper fibrations.
\end{example}

\begin{proof}
Let $d\in\dcl(Aa)$.
We wish to show that either $d\in\acl(A)$ or $a\in \acl(Ad)$.
Suppose some $a_i\in\acl(Ad)$.
Then, as every permutation of $\{a_1,\dots,a_r\}$ extends to an automorphism fixing $Aa$, and hence fixing $Ad$, it follows that each $a_1,\dots,a_r\in\acl(Ad)$.
But then $a\in\acl(Ad)$, and we are done.
We may therefore assume that each $a_i\notin\acl(Ad)$.
By triviality, it follows that  $d\ind_A(a_1,\dots,a_r)$.
But $d\in\acl(Aa_1\dots a_r)$, and hence $d\in\acl(A)$, as desired.
\end{proof}

The above example, while not itself minimal, is very much related to a minimal type.
The following proposition says that having no proper fibrations implies seminiminality.
Recall that  a stationary type $p=\tp(a/A)$ is semiminimal if and only if it is {\em almost internal} to a minimal type; that is, there exist $B\supseteq A$ with $a\ind_AB$, a minimal type $r$ over $B$, and a $B$-independent tuple $(e_1,\dots,e_\ell)$ of realisations of $r$, such that $\acl(Ba)=\acl(Be_1\dots e_\ell)$.

\begin{proposition}
\label{nofibrations}
Suppose $p=\tp(a/A)$ is stationary and admits no proper fibrations.
Then either
\begin{itemize}
\item[(i)]
$p$ is almost internal to a non locally modular minimal type, or
\item[(ii)]
 $a$ is interalgebraic over $A$ with a finite tuple of independent realisations of a locally modular minimal type over $A$.
\end{itemize}
In particular, $p$ is semiminimal.
\end{proposition}

\begin{proof}
We may assume that $p$ is nonalgebraic.
Hence $p$ is nonorthogonal to some minimal type $r$.
Let ${\bf R}$ be the set of $\acl(A)$-conjugates of $r$.
Then, as $p$ is not foreign to ${\bf R}$, there exists $c\in\dcl(Aa)\setminus\acl(A)$ with $\stp(c/A)$ ${\bf R}$-internal (see 7.4.6 of~\cite{pillay96}).
By the assumption of having no proper fibrations we must have $a\in\acl(Ac)$, so that $\stp(a/A)=p$ is almost ${\bf R}$-internal.
But this implies that $p$ is almost internal to $r$, and so $r$-semiminimal.

Suppose now that $r$ is locally modular.
Then, with $c$ as above,
we have that $\stp(c/A)$ is $1$-based.
If $U(c/A)=1$ then, as $a$ is interalgebraic with $c$ over $A$, we have that $p$ itself is minimal locally modular, and we are done.
If $U(c/A)>1$ then let $e$ be such that $U(c/Ae)=U(c/A)-1$, and $e=\cb(c/Ae)$.
By $1$-basedness, $e\in\acl(Ac)$, and hence the Lascar inequalities yield that
$$U(e/A)=U(c/A)+U(e/Ac)-U(c/Ae)=1$$
So $\stp(e/A)$ is locally modular and minimal.
Now let $(e_1,\dots,e_n)$ be the $Ac$-conjugates of $e$, so that $(e_1,\dots,e_n)\in\dcl(Ac)\subseteq \dcl(Aa)$.
The assumption of no fibrations yields that $a\in\acl(Ae_1,\dots,e_n)$.
So $(e_1,\dots,e_n)$ is a finite tuple of realisations of a locally modular minimal type over $A$, and $a$ is interalgebraic with this tuple over $A$.
\end{proof}

Case~(ii) of Proposition~\ref{nofibrations} naturally splits into two further cases; when the minimal locally modular type involved is trivial and when it is not trivial.
We have already seen a non minimal example (\ref{examplenofibre}) of the trivial case of~(ii).
The following construction witnesses that the non trivial case of~(ii) also occurs.

\begin{example}
Suppose $G$ is a locally modular strongly minimal group over $A=\acl(A)$ that is not $2$-torsion.
Let $E(x,y)$ be the equivalence relation $x=\pm y$ on $G$.
If $g,h$ are independent generic elements of $G$, and $a$ is a code for $\{g/E,h/E\}$, then $\tp(a/A)$ admits no proper fibrations.
\end{example}

\begin{proof}
Assume toward a contradiction that there exists a proper fibration.
Since $\tp(a/A)$ is of rank $2$, this means there exists $c\in\dcl(Aa)$ such that $\tp(a/Ac)$ is of rank $1$.
As the $E$-classes are finite, $g,h\in\acl(a)$.
Hence $p(x,y):=\stp(g,h/Ac)$ is stationary of rank $1$, and so by local modularity is the generic type of an $\acl(Ac)$-definable coset $X$ of a connected rank $1$ $A$-definable subgroup $H\leq G^2$.

\begin{claim}
\label{projH}
The co-ordinate projections $\pi_i:H\to G$, for $i=1,2$, are finite-to-one surjective group homomorphisms.
\end{claim}
\begin{proof}[Proof of Claim~\ref{projH}]
If not, by connectedness of $H$ and strong minimality, we may assume that $H=\{0\}\times G$.
So $X=\{g\}\times G$, and $g$ is the canonical parameter for $X$.
As $X$ is $\acl(Ac)$-definable, we get $g\in\acl(Ac)$.
Now, an $A$-automorphism taking $g$ to $h$ preserves $a$, and hence $c$.
So $h\in\acl(Ac)$ also.
This contradicts the fact that $p$ is of rank $1$.
\end{proof}

By the surjectivity of $\pi_1$, $X=(0,e)+H$ for some $e\in G$.
Let $\sigma$ be an $A$-automorphism of the universe fixing $g$ and taking $h$ to $-h$.

\begin{claim}
\label{conje}
$a\in\acl\big(Ae\sigma(e)\big)$.
\end{claim}
\begin{proof}[Proof of Claim~\ref{conje}]
Since $(g,h)\in(0,e)+H$, we have $(g,h-e)\in H$.
Hence $(g,-h-\sigma(e))\in H$ also.
We obtain that $\big(0,2h-e+\sigma(e)\big)$ and $\big(2g,-e-\sigma(e)\big)$ are both in $H$.
Since the co-ordinate projections are finite-to one, and since $G$ is not $2$-torsion, this implies that $g,h\in\acl\big(Ae\sigma(e)\big)$.
Hence $a\in\acl\big(Ae\sigma(e)\big)$.
\end{proof}

Note that $e\in\acl(Ac)$. Indeed, any automorphism fixing $\acl(Ac)$ preserves $X=(0,e)+H$ and $H$, and hence preserves $e+\ker(\pi_1)$, which is finite.
On the other hand, $\sigma$ fixes $a$ by definition, and hence fixes $c$.
So $\sigma(e)\in\acl(Ac)$ also.
Claim~\ref{conje} then implies that $a\in\acl(Ac)$, contradicting the assumption that $\tp(a/Ac)$ is of rank $1$.
\end{proof}

As for case~(i) of Proposition~\ref{nofibrations}, we will see later (in $\S$\ref{ccm-section} and $\S$\ref{dcf-section}) that non minimal examples appear in the theories $\ccm$ and $\dcf$.

\bigskip
\section{All Fibrations over $\PP$}
\label{allpfibre-section}

\noindent
Fix an $A$-invariant set of partial types, $\mathbb P$.
Often we are interested in the case when $\PP$ is the set of all non locally modular minimal types; in $\operatorname{DCF}_0$ we usually take $\PP$ to be the field of constants and in $\ccm$ it is the sort of the projective line.
But the material in this section makes sense for any $\PP$.

\begin{definition}
Suppose $p:=\tp(a/A)$ is stationary.
We say that {\em all fibrations of $p$ are over $\PP$} if whenever $c\in \dcl(Aa)$ with $a\notin\acl(Ac)$, then $\stp(c/A)$ is almost $\mathbb P$-internal.
\end{definition}

The generic type of a generalised hyperk\"ahler manifold in $\ccm$ has this property; this is Theorem~2.3(1) of~\cite{cop}.
See~\ref{example} below for examples in $\operatorname{DCF_0}$ coming from the logarithmic derivative on a simple abelian variety.

\begin{definition}[$\PP$-reduction]
\label{preduction}
Let $\Int_A(\PP):=\{c:\stp(c/A)\text{ is almost $\PP$-internal}\}$.
We say that $b$ is the {\em $\PP$-reduction of $a$ over $A$} if $\dcl(Ab)=\dcl(Aa)\cap\Int_A(\PP)$.
\end{definition}

\begin{remark}
\label{predrem}
\begin{itemize}
\item[(1)]
All fibrations of $\tp(a/A)$ being over $\PP$ can be reformulated as an ``exchange" property: if $c\in \dcl(Aa)\setminus\Int_A(\PP)$ then $a\in\acl(Ac)$.
\item[(2)]
The $\PP$-reduction of $a$ over $A$ is precisely the canonical base of $\tp\big(a/\Int_A(\PP)\big)$ and has the property that $\tp\big(a/\dcl(Aa)\cap\Int_A(\PP)\big)\vdash\tp\big(a/\Int_A(\PP)\big)$.
See, for example, Lemma~1 of the appendix to~\cite{chatzidakis-hrushovski99}.\end{itemize}
\end{remark}

We will characterise when $\tp(a/A)$ has all fibrations over $\PP$ in terms of the structure of the extension $\tp(a/Ab)$ where $b$ is the $\PP$-reduction of $a$ over $A$.
The key structural property in this characterisation will be the following natural notion of descent for types, which is the model-theoretic counterpart to birational descent in algebraic geometry and bimeromorphic descent in complex-analytic geometry (see Definition~\ref{descentfibre} below).

\begin{definition}[Descent for types]
\label{descent}
A complete stationary type $p\in S(B)$ {\em descends} to $B_0\subseteq\dcl(B)$ if there exist a stationary $q\in S(B_0)$ and a $B$-definable bijection $f:p^{\overline M}\to q_B^{\overline M}$, where by $q_B$ we mean the nonforking extension of $q$ to $B$.
If the map $f$ is only assumed to be finite-to-one and surjective, then we say that $p$ {\em almost descends} to $B_0$.
\end{definition}

\begin{remark}
\label{strongdescent}
\begin{itemize}
\item[(1)]
A special case of $p\in S(B)$ descending to $B_0\subseteq\dcl(B)$, is when $p$ does not fork over $B_0$ and $p\upharpoonright B_0$ is stationary.
\item[(2)]
Suppose $p=\tp(a/Ab)$ where $b$ is the $\PP$-reduction of $a$ over $A$.
If $p$ descends to $Ab_0$ witnessed by $q\in S(Ab_0)$, then in fact $q$ is weakly orthogonal to $\Int_A(\PP)$.
In particular, $q$ has a unique extension to $Ab$ and we have an $Ab$-definable bijection between $p^{\overline M}$ and $q^{\overline M}$.
\end{itemize}
\end{remark}

\begin{proof}[Proof of~\ref{strongdescent}(2)]
Let $q=\tp(a_0/Ab_0)$ with $a_0\ind_{Ab_0}b$.
Since $a\ind_{Ab}\Int_A(\PP)$ and $a_0\in\dcl(Aba)$, we have $a_0\ind_{Ab}\Int_A(\PP)$.
By transitivity we get that $a_0\ind_{Ab_0}\Int_A(\PP)$.
This implies, by invariance, that $q$ is weakly orthogonal to $\Int_A(\PP)$.
Using the fact that $b$ is from $\Int_A(\PP)$, we get the ``in particular" clause.
\end{proof}

Here is our characterisation of having all fibrations over $\PP$.

\begin{theorem}
\label{characterisepfibre}
Suppose $p=\tp(a/A)$ is stationary and not almost internal to~$\PP$, and $b$ is the $\PP$-reduction of $a$ over $A$.
Then the following are equivalent:
\begin{itemize}
\item[(i)]
all fibrations of $p$ are over $\mathbb P$,
\item[(ii)]
$\tp(a/Ab)$ admits no proper fibrations and does not almost descend to any relatively algebraically closed proper subset of $\dcl(Ab)$ containing $A$.
\end{itemize}
\end{theorem}

\begin{proof}
We begin with the following characterisation for descent of the $\PP$-reduction.

\begin{claim}
\label{descentequiv}
$\tp(a/Ab)$ almost descends to $Ab_0$ if and only if there exists $a_0\in\dcl(Aa)$ such that $a\in\acl(Aba_0)$ and 
$\dcl(Ab_0)$ contains the $\PP$-reduction of $a_0$ over $A$.
  \end{claim}
  
\begin{proof}[Proof of Claim~\ref{descentequiv}]
If $a_0\in\dcl(Aa)$ and $a\in\acl(Aba_0)$ then there exists an $Ab$-definable finite-to-one surjection $\tp(a/Ab)^{\overline M}\to \tp(a_0/Ab)^{\overline M}$.
If, moreover, $\dcl(Ab_0)$ contains the $\PP$-reduction of $a_0$ over $A$, then, as $b$ is from $\Int_A(\PP)$, we have $a_0\ind_{Ab_0}b$ and $\tp(a_0/Ab_0)$ is stationary.
Hence $\tp(a/Ab)$ almost descends to $Ab_0$, as desired.

Conversely, suppose $q\in S(Ab_0)$ and $f:\tp(a/Ab)^{\overline M}\to q_{Ab}^{\overline M}$ witnesses that $\tp(a/Ab)$ almost descends to $Ab_0$.
Letting $a_0:=f(a)$ it suffices to show that $b_0$ contains the $\PP$-reduction of $a_0$ over $A$.
We know that $a_0\ind_{Ab_0}b$ since $q_{Ab}$ is the nonforking extension of $q$.
On the other hand, as $b$ is the $\PP$-reduction of $a$ over $A$, $a_0\ind_{Ab}\Int_A(\PP)$.
Hence $a_0\ind_{Ab_0}\Int_A(\PP)$, and so $\dcl(Ab_0)$ contains the $\PP$-reduction of $a_0$ over $A$.
\end{proof}

Next, we prove that~(i) implies~(ii) in Theorem~\ref{characterisepfibre}.
It follows more or less immediately from the definitions that since all fibrations of $p=\tp(a/A)$ are over $\PP$, and $b$ is the $\PP$-reduction of $a$ over $A$, that $\tp(a/Ab)$ admits no proper fibrations.
So assuming that $\tp(a/Ab)$ descends to $Ab_0$ it remains for us to show that $b\in\acl(Ab_0)$.
Let $a_0$ be as given by Claim~\ref{descentequiv}.
Since $p$ is assumed to not be almost internal to $\PP$ and $a\in\acl(Aba_0)$, we must have $a_0\notin\Int_A(\PP)$.
By (i), $a\in\acl(Aa_0)$ and hence $b\in\acl(Aa_0)$.
But $a_0\ind_{Ab_0}\Int_A(\PP)$ by the claim, so that $b\in\acl(Ab_0)$, as desired.

Conversely, suppose that~(ii) holds and $c\in\dcl(Aa)\setminus\Int_A(\PP)$.
We need to show that $a\in\acl(Ac)$.
As $\tp(a/Ab)$ has no proper fibrations and $c\notin\acl(Ab)$, we must have $a\in\acl(Abc)$.
Let $b_0$ be the $\PP$-reduction of $c$ over $A$.
By Claim~\ref{descentequiv} we have that $\tp(a/Ab)$ almost descends to $Ab_0$, and so by~(ii), $b\in\acl(Ab_0)\subset\acl(Ac)$.
So $a\in\acl(Ac)$, as desired.
\end{proof}

\bigskip
\section{Consequences for $\ccm$}
\label{ccm-section}

\noindent
In this section we discuss the above notions in the particular case of the theory $\ccm$, from which, in any case, the ideas in this paper stem.
As we shall see, we recover the results of~$\S$2 of~\cite{cop} and even add a little to them.

First some preliminaries.
By a {\em fibration}  we will mean a dominant meromorphic map $f:X\to S$ between irreducible compact complex spaces whose general fibres are irreducible.
This differs slightly from standard terminology in that we insist on irreducibility rather than connectedness (see~\cite{cop}, for example).
Up to taking a normalisation of $X$ these notions will agree; and the reason for our stricter definition is the following characterisation.

\begin{fact}[Lemmas~2.7 and 2.11 of~\cite{ret}]
\label{stationarityinccm}
Work in a saturated model of $\ccm$.
Suppose $p=\tp(a/b)$, $X=\loc(a,b)$, $Y=\loc(b)$, and $\pi:X\to Y$ is the co-ordinate projection.
Then the following are equivalent:
\begin{itemize}
\item[(i)]
The general fibres of $\pi$ are irreducible.
\item[(ii)]
$p$ is stationary.
\item[(iii)]
The generic fibre $X_b$ is absolutely irreducible.
\end{itemize}
\end{fact}

Fibrations are well-behaved with respect to base change.
If $f:X\to S$ is a fibration and $g:T\to S$ is another dominant meromorphic map between irreducible compact complex spaces then the Zariski closure of
$$\{(a,b):\text{ $f$ is defined at $a$, $g$ is defined at $b$, and $f(a)=g(b)$}\}$$
in $X\times T$ has a unique irreducible component that projects onto $T$.
The projection of this component onto $T$ is a fibration whose general fibres agree with that of $f$.
We will denote this fibration by $X_{(T)}\to T$, and refer to it as the {\em strict pull back} of $X\to S$ in $X\times_ST\to T$.
Note that $X_{(T)}$ also projects onto $X$ and is maximal dimensional in $X\times_ST$.

In~\cite{cop} a fibration $f:X\to B$ is called {\em minimal} if whenever there is a factorisation
$$\xymatrix{
X\ar[d]_f\ar[rd]\\
B & Y\ar[l]
}$$
then either $\dim Y=\dim B$ or $\dim Y=\dim X$.
If $X=\loc(a)$ and $B=\loc(b)$, then it is not hard to see that $f$ is minimal if and only if $p=\tp(a/b)$ admits no proper fibrations in the sense of~$\S$\ref{nofibre-section}.
Our Proposition~\ref{nofibrations} then implies the following uniform version of Theorem~2.4(1) of~\cite{cop}.

\begin{proposition}
\label{minimalccm}
Suppose $X$ is an irreducible compact complex space of K\"ahler-type and $f:X\to B$ is minimal.
Then either
\begin{itemize}
\item[(I)]
the general fibres of $f$ are Moishezon; or,
\item[(II)]
the general fibres of $f$ are {\em uniformly isoptypically semi-simple} of algebraic dimension zero; that is, there is a commutative diagram
$$\xymatrix{
X\ar[dr]_f&Z\ar[l]_{\ p_X}\ar[r]^{p_Y}\ar[d] & Y\ar[dl]\\
&B
}$$
where $Y$ is of the form $Y'\times_BY'\times_B\cdots\times_BY'$ for some fibration $Y'\to B$ with general fibres simple of algebraic dimension zero, and $p_X,p_Y$ are generically finite surjective holomorphic maps.
\end{itemize}
\end{proposition}

\begin{proof}
This is essentially a translation of~\ref{nofibrations} which tells us that one of two things can happen:
One possibility is that the generic type of the generic fibre of $f$ is $r$-semiminimal for some non locally modular minimal type $r$.
As a consequence of the manifestation of the Zilber dichotomy in $\ccm$, this will imply that the general fibres of $f$ are Moishezon (see Fact~3.1 of~\cite{cbp} or Proposition~4.4 of~\cite{ret}).
The other possibility is that the generic type of the generic fibre of $f$ is interalgebraic with a finite tuple of (independent) realisations of a minimal locally modular type.
We deduce from this the existence of the diagram in~(II) where $Y'\to B$ has the property that the generic type of the generic fibre is minimal locally modular.
As we are working in the K\"ahler case we have ``essentially saturation" (see~\cite{sat}), and so the {\em general} fibres of $Y'\to B$ will have minimal locally modular generic types.
Minimality yields simplicity of the general fibres, and modularity forces them to have algebraic dimension zero (see Remark~3.3 of~\cite{pillay2000}).
\end{proof}

In the second part of Theorem~2.4 of~\cite{cop} the authors use a structure theorem of Fujiki's (Theorem~1 of~\cite{fujiki83a}) on the algebraic reduction of compound Moishezon spaces of K\"ahler-type to conclude further that, in the case when $f$ is the algebraic reduction, case~(I) entails that the general fibre of $X\to B$ is an abelian variety.\footnote{In the statement of Theorem~2.4(2) of~\cite{cop} the authors fail to mention the assumption that $f$ is the algebraic reduction. However, it is not hard to find counterexamples if we drop this assumption, and inspecting the proof shows that the authors have in mind the case when $X$ is hyperk\"ahler nonalgebraic, in which case $f$ is the algebraic reduction.}
We can make model-theoretic sense of this too.

\begin{proposition}
\label{moish-abelian}
Let $X$ be a non-Moishezon irreducible compact complex space of K\"ahler-type, and $f:X\to B$ the algebraic reduction.
Suppose that $f$ is minimal.
If the general fibre of $f$ is Moishezon then it is in fact an abelian variety.
\end{proposition}

\begin{proof}
We use the theory of internality and follow roughly the ideas of the second author in~\cite{pillay2000} (Fact~5.1).
Since we are in a K\"ahler-type space, we can work in a fixed full countable language for $X$ with respect to which our structure is saturated (see~\cite{sat}).
Let $a$ be generic in $X$, $b:=f(a)$, and $V$ the set of realisations of $\tp(a/b)$.
That the general fibre of $f$ is Moishezon implies that $V$ is internal to $\PP(\CC)$, and hence by the theory of internality (cf.~$\S$7.4 of~\cite{pillay96}) is $b$-definably and faithfully acted upon by a $b$-definable group $G$
that is definably isomorphic to a connected algebraic group.
The fact that $f$ is the algebraic reduction implies that $a\ind_b\PP(\CC)$, and hence the action is transitive.
Now $G$ has a unique maximal definable subgroup that is definably isomorphic to a connected linear algebraic group, let's call it $L$.
Note that $L$ must be $b$-definable and normal, and that the quotient $G/L$ is a $b$-definable group definably isomorphic to an abelian variety.

We first rule out the possibility that $L$ acts transitively on $V$.
If it did, then the generic fibre $X_b$ of $f$ would be almost homogeneous unirational.
This implies that $f$ is Moishezon (see Proposition~7 of~\cite{fujiki83}, this is where K\"ahler-type is used).
Since $B$ is Moishezon this would contradict $X$ being non-Moishezon.
So $L$ does not act transitively on $V$.

Next we argue that $L$ must fix $a$; we show that $L$ moving $a$ is ruled out by the minimality of $f$.
Let $c$ be a code for $L\cdot a$.
So $c\in\dcl(a)$.
Note that for any $a'\in V$, since the type of $a'$ and $a$ over $b$ agree, the code for $L\cdot a'$ has the same type as $c$ over $b$.
Since $L$ does not act transitively it will have infinitely many orbits in $V$ (by stationarity), and so $c\notin\acl(b)$.
On the other hand, note that if $a'\in L\cdot a$ then $a'$ and $a$ have the same type over $bc$.
If $L$ does not stabilise $a$ the orbit $L\cdot a$ is infinite (by connectedness), and hence $a\notin\acl(bc)$.
But $\tp(a/b)$ has no proper fibrations as $f$ is minimal, and that rules out the possibility of such a $c$.

It must therefore be the case that $L$ fixes $a$.
Since $V$ is a complete type over $b$ and everything is $b$-definable, $L$ must stabilise all of $V$.
By the faithfulness of the action, $L$ is thus trivial, and $G$ is definably isomorphic to some abelian variety $A\subseteq\PP^n(\CC)$.
We thus have a definable transitive action of $A$ on $V$.

It remains to show that this action of $A$ on $V$ is holomorphic.
Indeed, note first of all that $V$ is a Zariski open subset of $X_b$; this follows from saturation because $V$ is a definable set (it is the orbit of $a$ under the action of $A$) and also an intersection of Zariski open subsets of $X_b$ (as $\tp(a/b)$ is generic in $X_b$).
Hence the question of whether the action is holomorphic or not makes sense.
If it were holomorphic then by quantifier elimination there is a meromorphic map $\phi:A\times X_b\to X_b$ that agrees with the action of $A$ on $A\times V$.
But as the action is transitive, it follows that $V=\phi(A\times\{a\})$ is Zariski closed in $X_b$ and hence equal to $X_b$. So we have a transitive holomorphic action of $A$ on $X_b$, which forces $X_b$ to be an abelian variety.

So we prove that the action of $A$ on $V$ is holomorphic.
First of all, because of the commutativity of $A$, the parameters over which $A$ and its action on $V$ are defined can be taken to be a tuple $b'\supseteq b$ from $\PP(\CC)$.
Now, by quantifier elimination, the action restricted to some nonempty $b'$-definable Zariski open subset $U\subseteq A\times V$ is holomorphic.
But as $\tp(a/b)\vdash\tp(a/\PP)$ and both $b'$ and $A$ live in $\PP(\CC)$, this $U$ can be taken to be of the form $A'\times V$ where $A'$ is a nonempty $b'$-definable Zariski open subset of $A$.
It follows that for any $x\in A'$, the action of $A$ restricted to $(x+A')\times V$ is holomorphic as it is given by $(x+y,v)\mapsto x\cdot(y\cdot v)$.
But such translates of $A'$ by elements of $A'$ cover all of $A$.
So our action of $A$ on $V$ is holomorphic, as desired.
 \end{proof}

Next we consider the complex-analytic content of Theorem~\ref{characterisepfibre}.
The role of $\mathbb P$ here is played by the projective line sort.
The first thing to notice is that $\PP$-reductions in the sense of Definition~\ref{preduction} when specialised to $\ccm$ agree with algebraic reductions.
That is, 
{\em $b$ is the $\PP$-reduction of $a$ if and only if $b=f(a)$ where $f$ is the algebraic reduction map on $X:=\loc(a)$.}
Indeed, this follows from the fact that $\tp(b)$ is almost internal to $\PP$ if and only if $\loc(b)$ is Moishezon -- see Fact~3.1 of~\cite{cbp}, for example.

We now recall bimeromorphic descent for fibrations.

\begin{definition}[Descent for fibrations]
\label{descentfibre}
Suppose $h:X\to T$ is a fibration and $g:T\to S$ is a dominant meromorphic map.
We say that $h$ {\em descends to $S$} if there exists a fibration $\hat h:\widehat X\to S$, such that $X$ is bimeromorphically equivalent to $\widehat X_{(T)}$ over $T$.
In diagrams
$$\xymatrix{
X\ar[r]^{\approx\ \ \ }\ar[rd]_h & \widehat X_{(T)}\ar[d] & \widehat X\ar[d]^{\hat h}\\
& T\ar[r]_g&S
}$$
If instead of a bimeromorphic equivalence we only have that $X$ admits a generically finite dominant meromorphic map to $\widehat X_{(T)}$, then we say that $h$ {\em almost descends} to~$S$.
\end{definition}

This specialises the notion of descent for types introduced in Definition~\ref{descent}.

\begin{lemma}
\label{descentalgred}
Suppose $X=\loc(a)$, $h:X\to T$ is a fibration and $g:T\to S$ is a dominant meromorphic map.
The following are equivalent:
\begin{itemize}
\item[(i)]
$h$ almost descends to $S$,
\item[(ii)]
$\tp\big(a/h(a)\big)$ almost descends to $g\big(h(a)\big)$. 
\end{itemize}
\end{lemma}

\begin{proof}
Let $b:=h(a)$ and $b_0:=g(b)\in\dcl(b)$.
If $\tp(a/b)$ almost descends to $b_0$, then there is $a_0\in\dcl(a)$ such that $a\in\acl(ba_0)$, $a_0\ind_{b_0}b$, and $\tp(a_0/b_0)$ is stationary. Let $\widehat X:=\loc(a_0,b_0)$ and $\hat h:\widehat X\to S$ be the second co-ordinate projection.
By stationarity, $h$ is a fibration, and by independence, $\widehat X_{(T)}$ is the locus of $\tp(a_0,b)$.
The fact that $a_0\in\dcl(a)$ and $a\in\acl(ba_0)$ implies therefore, that there is a generically finite meromorphic map from $X$ to $\widehat X_{(T)}$ over $T$, as desired.
 
For the converse, assume that $h$ almost descends to $S$ witnessed by $\hat h:\widehat X\to S$ and a generically finite surjective map $f:X\to \widehat X_{(T)}$ over $T$.
Let $a_0\in\widehat X_{b_0}$ be such that $f(a)=(a_0,b)$.
Note that $a_0$ is then independent from $b$ over $b_0$, and $\tp(a_0/b_0)$ is stationary by Fact~\ref{stationarityinccm}.
Now $f$ restricts to a finite-to-one surjective $b$-definable map from the realisations of $\tp(a/b)$ to that of $\tp(a_0/b)$, as desired.
\end{proof}

Using this lemma, Theorem~\ref{characterisepfibre} readily specialises to:

\begin{theorem}
\label{characterisepfibreccm}
Suppose $X$ is a non-Moishezon, irreducible, compact complex space.
Then the following are equivalent:
\begin{itemize}
\item[(i)]
Whenever $X\to Y$ is a dominant meromorphic map, either $Y$ is Moishezon or $\dim Y=\dim X$.
\item[(ii)]
The algebraic reduction map $X\to B$ is minimal and does not almost descend to any $S$ with $\dim S<\dim B$.
\end{itemize}
\end{theorem}

The following corollary thus recovers most of what is done in~$\S2$ of~\cite{cop}, but with additional uniformity and the new observation about descent.

\begin{corollary}
\label{hyperkaehler}
If $X$ is a nonalgebraic generalised hyperk\"ahler manifold then the algebraic reduction $X\to B$ is minimal and does not almost descend to any $S$ with $\dim S<\dim B$.
Moreover, either the general fibre of $X\to B$ is an abelian variety or it is uniformly isotypically semi-simple of algebraic dimension $0$ (that is, {\em (II)} of Proposition~\ref{minimalccm} holds).
\end{corollary}

\begin{proof}
By 2.3(1) of~\cite{cop}, $X$ satisfies condition~(i) of Theorem~\ref{characterisepfibreccm}, and hence also condition~(ii).
The moreover clause is by Propositions~\ref{minimalccm} and~\ref{moish-abelian}.
\end{proof}

\bigskip
\section{Consequences for $\dcf$}
\label{dcf-section}

\noindent
In this section we specialise the results of $\S$\ref{allpfibre-section} to the context of differentially closed fields to obtain differential-algebraic geometric statements.
As the proofs are very much analogous to the complex geometric case dealt with above, we leave them out entirely.
We will illustrate the theorems with an example.
 
We work throughout in a sufficiently saturated model $(K,\delta)\models\dcf$ and over an algebraically closed $\delta$-subfield $k\subset K$.
The results of $\S$\ref{allpfibre-section} apply only to the stationary types of finite $U$-rank, and these are precisely the generic types of finite dimensional irreducible $\delta$-varieties over $k$.
Here, for $X$ an irreducible $\delta$-variety over~$k$, by {\em finite dimensional} we mean that the $\delta$-rational function field $k\langle X\rangle$ is of finite transcendence degree over $k$, and we denote this value by $\dim_\delta X$.

The role of $\mathbb P$ will be played by $\mathcal C$, the field of constants of $K$.
Suppose $X$ is a finite dimensional irreducible $\delta$-varieties over $k$ with $p(x)\in S(k)$ its generic type.
For $p(x)$ to be almost internal to $\C$ is equivalent to $X$ being in generically finite-to-finite correspondence with an algebraic variety in the constants.
More precisely, we say that $X$ is {\em $\C$-algebraic} if  there exists an irreducible algebraic variety $V$ over $\C$ and an irreducible $\delta$-variety $\Gamma\subset X\times V(\C)$ that projects generically finite-to-one onto both $X$ and $V(\C)$.
Note that $V$ and $\Gamma$ do not need to be defined over $k$.

\begin{remark}
Unlike the case of $\ccm$, being $\C$-algebraic is not equivalent to being $\delta$-birationally equivalent to $V(\C)$ for some algebraic variety $V$ over $\C$.
\end{remark}

By a {\em fibration of $X$ over $k$} we mean a $\delta$-variety $B$ over $k$ together with a $\delta$-dominant $\delta$-rational map $f:X\to B$ over $k$ whose generic fibre is (absolutely) irreducible.
A fibration is {\em minimal} if it does not factor through any proper fibration $X\to Z$ with $\dim_\delta X>\dim_\delta Z>\dim_\delta B$.
This is equivalent to the generic type of the generic fibre admitting no proper fibrations in the sense of Definition~\ref{nofibrationsdef}.
We have the following specialisation of Proposition~\ref{nofibrations} to $\dcf$.

\begin{proposition}
\label{minimaldcf}
Suppose $X$ is an irreducible finite dimensional $\delta$-variety over $k$ and $f:X\to B$ is a minimal fibration.
Let $b\in B$ be generic over $k$.
Then either
\begin{itemize}
\item[(I)]
$X_b$ is $\C$-algebraic; or,
\item[(II)]
$X_b$ is in generically finite-to-finite $\delta$-rational correspondence over $k\langle b\rangle$ with a $\delta$-variety of the form $Y^n$ where $Y$ is an irreducible $\delta$-variety over $k\langle b\rangle$ whose generic type is minimal and locally modular.
\end{itemize}
\end{proposition}

Descent in the differential-algebraic geometric context takes the following form:
Given a fibration $f:X\to B$ and a $\delta$-dominant map $g:B\to S$ over $k$, we say that $f$ {\em almost descends to $S$} if there exists a fibration $\hat f:\widehat X\to S$ over $k$, such that $X$ admits a generically finite-to-one $\delta$-rational map onto $\widehat X_{(B)}$ over $B$.
Here $\widehat X_{(B)}$ is the irreducible component of $\widehat X\times_SB$ that projects $\delta$-dominantly onto $B$.

The {\em $C$-algebraic reduction of $X$ over $k$} is an irreducible $\C$-algebraic $\delta$-variety $B$ together with a fibration $f:X\to B$ over $k$, such that whenever $g:X\to Y$ is $\delta$-dominant $\delta$-rational over $k$ and $Y$ is $\C$-algebraic, then $g$ factors through $f$.
This is unique up to $\delta$-birational equivalence.
Note that if $a\in X$ is generic over $k$ then $\dcl\big(kf(a)\big)=\dcl(ka)\cap\Int_k(\C)$, so that $\C$-algebraic reductions are just the $\delta$-algebraic geometric manifestation of Definition~\ref{preduction}.

Theorem~\ref{characterisepfibre} specialises to:

\begin{theorem}
\label{characterisepfibredcf}
Suppose $X$ is an irreducible finite dimensional $\delta$-variety over $k$ that is not $\C$-algebraic.
Then the following are equivalent:
\begin{itemize}
\item[(i)]
All fibrations of $X$ are over $\C$.
That is, if $X\to Y$ is a fibration with $\dim_\delta Y<\dim_\delta X$, then $Y$ is $\C$-algebraic.
\item[(ii)]
The $\C$-algebraic reduction $X\to B$ is minimal and does not almost descend to any $S$ with $\dim_\delta S<\dim_\delta B$.
\end{itemize}
\end{theorem}

We conclude with an example that witnesses the non vacuity of Theorem~\ref{characterisepfibredcf}.

\begin{example}
\label{example}
{\em Suppose $k\subset\C$ and fix a simple abelian variety $A$ over $k$.
The logarithmic-derivative $\ell:A(K)\to T_0A(K)$ is a $\delta$-rational surjective homomorphism from $A$ to its Lie algebra, with kernel $A(\C)$.
(See $\S$3 of~\cite{markermanin} for details.)
Identifying $T_0A(K)$ with $\mathbb G^d_{\text{a}}(K)$, let $G:=\ell^{-1}\big(\mathbb G_{\text{a}}(\C)\times\{0\}^{d-1}\big)$ and set $\pi:=\ell|_G$.
We then have the following short exact sequence of connected finite dimensional $\delta$-algebraic subgroups of $A(K)$ over $k$,}
$$\xymatrix{
0\ar[r]&A(\C)\ar[r]&G\ar[r]^{\pi \ \ \ \ }&\mathbb G_{\text{a}}(\C)\ar[r]&0
}$$
Then $\pi$ is the $\C$-algebraic reduction of $G$, it is a minimal fibration, and it does not almost descend.
In particular, by Theorem~\ref{characterisepfibredcf}, all fibrations of $G$ are over $\C$.
\end{example}

\begin{proof}
We first show that $G$ is not $\C$-algebraic.
Indeed, if it were then in fact it would be definably isomorphic to the $\C$ points of some algebraic group over $\C$ -- see Corollary~3.10 of~\cite{pillay06} -- and hence would have to be isomorphic to $A(\C)\times\mathbb G_{\text{a}}(\C)$ as there are in algebraic geometry no nontrivial extensions of the additive group by an abelian variety.
So this gives rise to a Zariski-dense (by the simplicity of $A$) torsion-free $\delta$-definable subgroup of $A(K)$, namely $\mathbb G_{\text{a}}(\C)$.
This is impossible as any Zariski dense $\delta$-definable subgroup of a commutative and connected algebraic group must contain the torsion of that algebraic group (see Corollary~4.2 of~\cite{pillaydag}).

Let $p$ be the generic type of a generic fibre of $\pi$.
That is, $p(x)=\tp(a/kb)$ where $a\in G$ is generic over $k$ and  $b=\pi(a)$.
Note that $p$ is stationary and $\C$-internal since the fibres of $\pi$ are cosets of $A(\C)$.
We claim that 

\begin{claim}
\label{exampleclaim}
$A(\C)$ together with its action on $p(K)$ is (up to definable isomorphism) the binding group $\aut\big(p(K)/\C\big)$.
In particular, $p(K)=a+A(\C)$.
\end{claim}

\begin{proof}[Proof of~\ref{exampleclaim}]
Let $H=\aut\big(p(K)/\C\big)$.
We show that $\phi:H\to A(\C)$ given by $\phi(h)=a-ha$ is an isomorphism.
Given $h,h'\in H$ we have
\begin{eqnarray*}
hh'a
&=&
h\big(a+\phi(h')\big)\\
&=&
\sigma_h\big(a+\phi(h')\big)\ \ \ \ \ \text{ where $\sigma_h\in\aut_{\C}(K)$ extends the action of $h$} \\
&=&
\sigma_h(a)+\phi(h') \ \ \ \ \ \text{ as $\phi(h')$ is a $\C$-point}\\
&=&
a+\phi(h)+\phi(h')
\end{eqnarray*}
showing that $\phi$ is a group homomorphism.
Next, suppose $ha=a$.
Then for any $a'\models p(x)$ we have $a'=a+c$ for some $c\in A(\C)$, so that $ha'=\sigma_h(a+c)=a+c=a'$.
This shows that $\phi$ is injective.
By simplicity of $A$, either $H$ is trivial or $\phi(H)=A(\C)$.
The former would imply that  $a\in G(\C)$, contradicting the fact that $G$ is not $\C$-algebraic.
Hence $\phi$ is surjective.
\end{proof}

One consequence of Claim~\ref{exampleclaim} is that the binding group acts transitively on $p(K)$.
Hence $a\ind_{kb}\C$.
It follows that $a\ind_{kb}\Int_k(\C)$ and so $kb$ is the canonical base of $\tp\big(a/\Int_k(\C)\big)$.
This proves that $\pi$ is the $\C$-algebraic reduction of $G$.

Now suppose that $\pi$ factors through a fibration $f:G\to Y$, and let $q(x)=\tp\big(a/kf(a)\big)$.
Then by Claim~\ref{exampleclaim}, $F=\aut\big(q(K)/\C\big)$ is a connected $\delta$-definable subgroup of $A(\C)$. 
If $F=A(\C)$ then $q(K)=p(K)$ and $Y\to\mathbb G_{\text{a}}(\C)$ is $\delta$-birational.
Otherwise, $F$ is trivial and $a\in\dcl(kf(a)\C)$.
But as $a\ind_{kf(a)}\C$, this forces $f$ to be $\delta$-birational.
We have shown that $\pi$ is a minimal fibration.

Finally, it remains to show that $\pi$ does not descend with respect to any $\delta$-dominant $\mathbb G_{\text{a}}(\C)\to S$ with $\dim_\delta S<\dim_\delta\mathbb G_{\text{a}}(\C)=1$.
But such an $S$ would then have to be a point, and descent would in this case mean that $\pi$ factors through a generically finite-to-one $\delta$-rational map $G\to \mathbb G_{\text{a}}(\C)\times \widehat G$.
It would follow that $A(\C)$ is a generically finite cover of $\widehat G$, making the latter $\C$-algebraic, and thereby contradicting the non-$\C$-algebraicity of $G$.
\end{proof}

\bigskip
\section{An aside on maximal covering families}
\label{appendix}

\noindent
As we saw in~$\S$\ref{ccm-section}, the specialisation of Proposition~\ref{nofibrations} to the theory $\ccm$ yields a uniform version of Theorem~2.4(1) of~\cite{cop}.
However, our proof was not abstracted from that in~\cite{cop}.
The argument there does not go via a minimal locally modular type.
Instead, the authors establish first a certain fact about maximal covering families of algebraic dimension zero compact K\"ahler manifolds (Lemma~2.7 of~\cite{cop}); namely that any such family is parametrised by a simple manifold, and only finitely many of the members of the covering family pass through a general point.
This result seems to be of independent interest, and is it turns out, can be given a model-theoretic explanation as follows.

\begin{proposition}
\label{maxfam}
{Suppose $T$ is the theory of a Zariski structure\footnote{This is in the sense of Zilber~\cite{zilberbook}.} with the $\cbp$\footnote{This is the``canonical base property", see~\cite{cbp} for details.}, and $p=\stp(a/A)$ is orthogonal to the set of non locally modular minimal types.
Let $q=\stp(a/E)$ be a forking extension of $p$ with maximal locus.\footnote{The {\em locus} of $\tp(a/E)$ is the smallest Zariski-closed set defined over $E$ containing $a$.}
Let $e=\cb(q)$.
Then $e\in\acl(Aa)$ and $U(e/A)=1$.}
\end{proposition}

\begin{proof}
First we note, without using the $\cbp$, that $q$ being a forking extension of $p$ of maximal locus is equivalent to $U(q)=U(p)-1$.
The right to left direction is clear.
For the converse,
note that $e\notin\acl(A)$ and hence there is a $B\supset A$ such that $U(e/B)=1$.
Let $q'=\stp(a'/Be)$ be the nonforking extension of $q$ to $Be$.
Since $e=\cb(q')$ and $e\notin\acl(B)$, $a'\nind_Be$.
That is, $e\in\acl(Ba')$.
By Lascar inequalities it follows that $U(a'/B)=U(a'/Be)+1=U(q)+1$.
So if $U(q)<U(p)-1$ then $\stp(a'/B)$ is a forking extension of $p$ which has $q'$ as a forking extension.
The locus of $\stp(a'/B)$ thus properly contains $\loc(q')=\loc(q)$, which is a contradiction to maximality.

So $U(q)=U(p)-1$.
Now if $e\notin\acl(Aa)$, then letting $e_0=\acl(Ae)\cap\acl(Aa)$ we would have $a\nind_{e_0}e$ since $e=\cb(a/e)$.
So $U(a/e_0)>U(a/e)=U(p)-1$ which forces $a\ind_Ae_0$, and so $\stp(a/e_0)$ is still orthogonal to the set of non locally modular minimal types.
On the other hand, CBP tells us that $\stp(e/e_0)$ is almost internal to the set of non locally modular minimal types
(this is Theorem~2.1 of~\cite{chatzidakis12} but see also Proposition~4.4 of~\cite{cbp}).
This contradicts $a\nind_{e_0}e$.
Hence $e\in\acl(Aa)$.
That $U(e/A)=1$ now follows immediately from the Lascar inequalities.
\end{proof}


\end{document}